\documentclass[12pt]{article}
\usepackage{amsmath,amssymb,amsfonts,amsthm}
\newenvironment{myabstract}{\par\noindent
{\bf Abstract . } \small }
{\par\vskip8pt minus3pt\rm}
\newcounter{item}[section]
\newcounter{kirshr}
\newcounter{kirsha}
\newcounter{kirshb}
\newenvironment{enumroman}{\setcounter{kirshr}{1}
\begin{list}{(\roman{kirshr})}{\usecounter{kirshr}} }{\end{list}}
\newenvironment{enumarab}{\setcounter{kirshb}{1}
\begin{list}{(\arabic{kirshb})}{\usecounter{kirshb}} }{\end{list}}
\newenvironment{athm}[1]{\vskip3mm\par\noindent
{\bf #1 }. \slshape }
{\upshape\par\vskip10pt minus3pt}
\newtheorem{theorem}{Theorem}[section]

\newtheorem{lemma}[theorem]{Lemma}

\theoremstyle{definition}

\newtheorem{example}[theorem]{Example}
\newtheorem{definition}[theorem]{Definition}

\def\C{{\mathfrak{C}}}
\def\Fm{{\mathfrak{Fm}}}

\def\At{{\bf At}}

\def\Fr{{\mathfrak{Fr}}}
\def\Sg{{\mathfrak{Sg}}}
\def\Fm{{\mathfrak{Fm}}}
\def\A{{\mathfrak{A}}}
\def\B{{\mathfrak{B}}}
\def\C{{\mathfrak{C}}}

\def\M{{\mathfrak{M}}}
\def\N{{\mathfrak{N}}}

\def\CA{{\bf CA}}

\def\SC{{\bf SC}}

\def\RSC{{\bf RSC}}

\def\RCA{{\bf RCA}}

\def\Rd{{\ Rd}}
\def\(R)RA{{\bf (R)RA}}
\def\RA{{\bf RA}}

\def\Ax{{\sf Ax}}
\def\tr{{\sf tr}}

 \def\CA{{\sf CA}}
\def\B{{\sf B}}

\def\Ra{{\mathfrak{Ra}}}

\def\Ra{{\mathfrak{Ra}}}

\def\A{{\mathfrak{A}}}
\def\B{{\mathfrak{B}}}
\def\C{{\mathfrak{C}}}

\def\A{{\mathfrak{A}}}
\def\B{{\mathfrak{B}}}
\def\C{{\mathfrak{C}}}

\def\L{{\mathfrak{L}}}
\def\Rd{{\mathfrak{Rd}}}

\def\At{{\mathfrak{At}}}
\def\L{{\mathfrak{L}}}
\def\Bl{{\mathfrak{Bl}}}
\def\CA{{\bf CA}}
\def\RA{{\bf RA}}

\def\RCA{{\bf RCA}}

\def\At{{\sf{At}}}
\def\N{\mathbb{N}}

\def\Cs{{\bf Cs}}

\def\Ra{{\mathfrak{Ra}}}

\def\CA{{\bf CA}}
\def\RCA{{\bf RCA}}

\def\Sg{{\mathfrak Sg}}

\def\Rl{{\mathfrak Rl}}
\def\N{{\cal N}}

\def\At{{\sf At}}
\def\Ig{{\sf Ig}}

\def\QRA{{\sf QRA}}

\title{Free algebras in Boolean algebras with operators}
\author{Tarek Sayed Ahmed \\
Department of Mathematics, Faculty of Science,\\ 
Cairo University, Giza, Egypt.
  }
\begin{document}
\maketitle

\begin{myabstract} 
We study atomicity of free algebras in varieties of Boolean algebras with operations, 
and we give some applications to cylindric-like algebras, mostly simplifying existing proofs in the literature due to N\'emeti, Tarski and Givant.
We obtain a new result concerning Pinters algebra, namely that the free $n$ 
dimensional representable algebra with $m$ free generators, 
can be generated but not freely with a (redundant)  set of $m$ elements.
\end{myabstract}


Cylindric and relation algebras were introduced by Tarski to algebraize first order logic. 
The structures of free cylindric and relation algebras are quite rich 
since they are able to capture the whole of first order logic, in a sense.
One of the first things to investigate about these free algebras is whether they 
are atomic or not, i.e. whether their boolean reduct is atomic or not.
By an atomic boolean algebra we mean an algebra for which  
below every non-zero element there is an atom, i.e. a minimal non-zero 
element. Throughout $n$ will denote a countable cardinal (i.e. $n\leq \omega$). 
More often than not, $n$ will be 
finite.
$\CA_n$ stands for the class of cylindric algebras of dimension $n$.
For a class $K$ of algebras, and a cardinal $\beta>0$, 
$\Fr_{\beta}K$ stands for the $\beta$-generated free
$K$ algebra.
In particular, $\Fr_{\beta}\CA_n$ 
denotes the $\beta$-generated free cylindric algebra
of dimension 
$n$. The following is known:
If $\beta\geq \omega$, then $\Fr_{\beta}\CA_n$ is atomless (has no atoms)
 [Pigozzi \cite{HMT1} 2.5.13].
Assume that $0<\beta<\omega$. If $n <2$ then $\Fr_{\beta}\CA_n$ is finite,
hence atomic, 
\cite{HMT1} 2.5.3(i).
$\Fr_{\beta}\CA_2$ is infinite but still atomic [Henkin, \cite{HMT1} 2.5.3(ii), 2.5.7(ii).]
If $3\leq n<\omega$, then $\Fr_{\beta}\CA_n$ has infinitely many atoms
[Tarski, \cite{HMT1} 2.5.9], and it was posed as an open question, cf \cite{HMT1} 
problem 4.14, whether it is atomic or not.
Here we prove, as a partial solution of problem 4.14 in \cite{HMT1}, and among other things, 
that $\Fr_{\beta}\CA_n$ is not atomic for $\omega>\beta>0$ and $\omega>n\geq 4.$ 
Here we investigate atomicity or non atomicity of free algebras in (often discriminator varieties) of Boolean algebras with 
operators.


\section{Free algebras in a broad context}

In cylindric algebra theory, whether the free algebras are atomic or not is an important topic. In fact, N\'emeti proves
that for $n\geq 3$ the free algebras of dimension $n$ on a finite set of generators are not atomic,
and this is closely related to Godels incompleteness theorems for the finite $n$-variable fragments of first order logic.
We first start by proving slightly new results concerning free algebras of classes of $BAO$'s.

\begin{definition}
Let $K$ be variety  of $BAO$'s. Let $\L$ be the corresponding multimodal logic.
We say that $\L$ has the {\it Godel's incompleteness property} if there exists
a formula $\phi$ that cannot be extended to a recursive complete theory.
Such formula is called incompletable.
\end{definition}
Let $\L$ be a general modal logic, and let $\Fm_{\equiv}$ be the Tarski-Lindenbaum formula algebra on
finitely many generators.
\begin{theorem}(Essentially Nemeti's) If $\L$ has $G.I$, then the algebra $\Fm_{\equiv}$
is not atomic.
\end{theorem}
\begin{proof}
Assume that $\L$ has $G.I$. Let $\phi$ be an incompletable
formula. We show that there is no atom in the Boolean algebra $\Fm$
below $\phi/\equiv.$
Note that because $\phi$ is consistent, it follows that $\phi/\equiv$ is non-zero.
Now, assume to
the contrary that there is such an atom $\tau/\equiv$ for some
formula $\tau.$
This means that .
that $(\tau\land \bar{\phi})/\equiv=\tau/\equiv$.
Then it follows that
$\vdash (\tau\land \phi)\implies \phi$, i.e.
$\vdash\tau\implies \phi$.
Let
$T=\{\tau,\phi\}$
and let
$Consq(T)=\{\psi\in Fm: T\vdash \psi\}.$
$Consq(T)$ is short for the consequences of $T$.
We show that $T$ is complete and that $Consq(T)$ is
decidable.   Let $\psi$ be an arbitrary formula in $\Fm.$
Then either $\tau/\equiv\leq \psi/\equiv$ or $\tau/\equiv\leq \neg \psi/\equiv$
because $\tau/\equiv$ is an
atom. Thus $T\vdash\psi$ or $T\vdash \neg \psi.$
Here it is the {\it exclusive or} i.e. the two cases cannot occur together.
Clearly $ConsqT$ is recursively enumerable.  By completeness of $T$ we have
$\Fm_{\equiv}\smallsetminus Consq(T)=\{\neg \psi: \psi\in Consq(T)\},$
hence the complement of $ConsqT$ is recursively enumerable as well, hence $T$
is decidable.  Here we are using the trivial fact that $\Fm$ is decidable.
This contradiction proves that $\Fm_{\equiv}$ is not atomic.
\end{proof}
In the following theorem, we give a unified perspective 
on several classes of algebras, studied in algebraic logic. Such algebras are cousins of cylindric algebras; though
the differences, in many cases, can be subtle and big.

(1) holds for diagonal free cylindric algebras, cylindric algebras, Pinter's substitution algebras
(which are replacement algebras endowed with cylindrifiers) and quasipolydic algebras
with and without equality when the dimension is $\leq 2$. (2) holds for Boolean algebras; we do not know whether it extends any further.  
(3) holds for such algebras for all finite dimensions. 

In fact, (1) holds for any discriminator variety $V$ of $BAO$'s, with finitely many operators, 
when $V$ is generated by the discriminator class $SirK$, of subdirectly indecomposable algebras having a discriminator term. To prove the latter, 
we start by a (well-known) lemma:

\begin{lemma} Let $L\supseteq L_{BA}$ be a functional signature, and $V$ a variety of $L-BAO$'s. Let $d(x)$ be a unary $L$ term. 
Then the following are equivalent:
\begin{enumarab}
\item $d$ is a discriminator term of $SirV$, so that $V$ is a discriminator variety.
\item all equations of the following for are valid in $V$:
\begin{enumerate}
\item $x\leq d(x)$

\item $d(d(x))\leq d(x)$
\item $f(x)\leq d(x)$ for all $f\in L\sim L_{BA}$
\end{enumerate}
\end{enumarab}
\end{lemma}

\begin{theorem} \label{free}Let $K$ be a variety of Boolean algebras with finitely many operators.
\begin{enumarab}
\item Assume that  $K=V(Fin(K))$, and for any $\B\in K$ and $b'\in \B$, there exists a regular $b\in \B$ such that
$\Ig^{\B}\{b'\}=\Ig^{\Bl\B}\{b\}$. If $\A$ is finitely generated, then $\A$ is atomic, hence 
the finitely generated free algebras are atomic. In particular, if $K$ is a discriminator variety, with discriminator term 
$d$, then finitely generated algebras are 
atomic. (One takes $b'=d(b)$).

\item Assume That $V$ is a $BAO$ and that the condition above on principal ideals, together with the condition that
that if $b_1'$ and $b_2$'s are the generators of two given ideals happen to be a partition (of the unit), 
then $b_0, b_1$ can be chosen to be also a partition. Then
$\Fr_{\beta}K_{\alpha}\times \Fr_{\beta}K_{\alpha}\cong \Fr_{|\beta+1|}K.$ In particular if $\beta$ is infinite, and
$\A=\Fr_{\beta}K$, then $\A\times \A\cong \A$. 
\item Assume that $\beta<\omega$, and assume the above condition on principal ideals.
Suppose further that for every $k\in \omega$, there exists an algebra $\A\in K$, with at least $k$ atoms, 
that is generated by a single element. Then $\Fr_{\beta}K$ has infinitely many atoms.
\item  Assume that $K=V(Fin(K))$.
Suppose $\A$ is $K$ freely generated by a finite set $X$ and $\A=\Sg Y$ with $|Y|=|X|$. Then $\A$ is $K$ freely generated
by $Y.$
\end{enumarab}
\end{theorem}
\begin{proof}
\begin{enumarab}
\item Assume that $a\in A$ is non-zero. Let $h:\A\to \B$ be a homomorphism of $\A$ into a finite algebra $\B$ such that
$h(a)\neq 0$. Let $I=ker h.$ We claim that $I$ is a finitely generated ideal.
Let $R_I$ be the congruence relation corresponding to $I$, that is $R_I=\{(a,b)\in A\times A: h(a)=h(b)\}$.

Let $X$ be a finite set such that $X$ generates $\A$ and $h(X)=\B$. Such a set obviously exists.
Let $X'=X\cup \{x+y: x, y\in X\}\cup \{-x: x\in X\}\cup \bigcup_{f\in t}\{f(x): x\in X\}.$
Let $R=\Sg^{\A}(R_I\cap X\times X')$. Clearly $R$ is a finitely generated congruence and $R_I\subseteq R$.
We show that the converse inclusion also holds.

For this purpose we first show that $R(X)=\{a\in A: \exists x\in X (x,a)\in R\}=\A.$
Assume that $xRa$ and $yRb$, $x,y\in X$ then $x+yRa+b$, but there exists $z\in X$ such that $h(z)=h(x+y)$ and $zR(x+y)$, hence
$zR(a+b)$ , so that $a+b\in R(X)$. Similarly for all other operations. Thus $R(X)=A$.
Now assume that $a,b\in A$ such that $h(a)=h(b)$.
Then there exist $x, y\in X$ such that $xRa$ and $xRb$. Since $R\subseteq ker h$,
we have $h(x)=h(a)=h(b)=h(y)$ and so $xRy$, hence $aRb$ and $R_I\subseteq R$.
So $I=\Ig\{b'\}$ for some element $b'$.  Then there exists $b\in \A$ such that  $\Ig^{\Bl\B}\{b\}=\Ig\{b'\}.$ Since $h(b)=0$ and $h(a)\neq 0,$
we have $a.-b\neq 0$. If $a.-b=0$, then $h(a).-h(b)=0$

Now $h(\A)\cong \A/\Ig^{\Bl\B}\{b\}$ as $K$ algebras. Let $\Rl_{-b}\A=\{x: x\leq -b\}$. Let $f:\A/\Ig^{\Bl\B}\{b\}\to \Rl_{-b}\A$ be defined by
$\bar{x}\mapsto x.-b$. Then $f$ is an isomorphism of Boolean algebras (recall that the operations of $\Rl_{-b}\B$ are defined by
relativizing the Boolean operations to $-b$.)
Indeed, the map is well defined, by noting that if $x\delta y\in \Ig^{\Bl\B}\{b\}$, where $\delta$ denotes symmetric difference,
then $x.-b=y.-b$ because $x, y\leq b$.

Since $\Rl_{-b}\A$ is finite, and $a.-b\in \Rl_{-b}\A$ is non-zero, then there exists an atom $x\in \Rl_{-b}\A$ below $a$,
but clearly $\At(\Rl_{-b}\A)\subseteq \At\A$ and we are done.

\item Let $(g_i:i\in \beta+1)$ be the free generators of $\A=\Fr_{\beta+1}K$.
We first show that $\Rl_{g_{\beta}}\A$ is freely generated by
$\{g_i.g_{\beta}:i<\beta\}$. Let $\B$ be in $K$ and $y\in {}^{\beta}\B$.
Then there exists a homomorphism $f:\A\to \B$ such that $f(g_i)=y_i$ for all $i<\beta$ and $f(g_{\beta})=1$.
Then $f\upharpoonright \Rl_{g_{\beta}}\A$ is a homomorphism such that $f(g_i.g_{\beta})=y_i$. Similarly
$\Rl_{-g_{\beta}}\A$ is freely generated by $\{g_i.-g_{\beta}:i<\beta\}$.
Let $\B_0=\Rl_{g_{\beta}}\A$ and $\B_1=\Rl_{g_{\beta}}\A$.
Let $t_0=g_{\beta}$ and $t_1=-g_{\beta}$. Let $x_i$ be such that $J_i=\Ig\{t_i\}=\Ig^{Bl\A}\{x_i\}$, and $x_0.x_1=0$.
Exist by assumption. Assume that $z\in J_0\cap J_1$. Then $z\leq x_i$,
for $i=0, 1$, and so  $z=0$. Thus $J_0\cap J_1=\{0\}$. Let $y\in A\times A$, and let $z=(y_0.x_0+y_1.x_1)$, then $y_i.x_i=z.x_i$ for each $i=\{0,1\}$
and so $z\in \bigcap y_0/J_0\cap y_1/J_1$. Thus $\A/J_i\cong \B_i$, and so
$\A\cong \B_0\times \B_1$.

\item Let $\A=\Fr_{\beta}K.$ Let $\B$ have $k$ atoms and generated by a single element. Then there exists a surjective
homomorphism $h:\A\to \B$. Then, as in the first item,  $\A/\Ig^{\Bl\B}\{b\}\cong \B$, and so $\Rl_{b}\B$ has $k$ atoms.
Hence $\A$ has $k$ atoms for any $k$ and we are done.

\item Let $\A=\Fr_XK$, let $\B\in Fin(K)$ and let $f:X\to \B$. Then $f$ can extended to a homomorphism $f':\A\to \B$.
Let $\bar{f}=f'\upharpoonright Y$. If $f, g\in {}^XB$ and $\bar{f}=\bar{g}$,
then $f'$ and $g'$ agree on a generating set $Y$, so $f'=g',$ hence $f=g$.
Therefore we obtain a one to one mapping from $^XB$ to $^YB$, but $|X|=|Y|,$
hence this map is surjective. In other words for each $h\in {}^YB,$ there exists a unique
$f\in {}^XB$ such that $\bar{f}=h$, then $f'$ with domain $\A$ extends $h.$
Since $\Fr_XK=\Fr_X(Fin(K))$ we are done.
\end{enumarab}
\end{proof}

\section{Two new results on Pinter's algebras}

\begin{example}

Let $U, n$ be finite, such that each has at least two elements, and $n>2$. Let $\B= \wp({^nU})\in \SC_n$,  
$X=\{s\in {}^nU: s_0<s_1\}$, and $\A=\Sg^{\A}\{X\}$. Define by recursion, $Y_0=^nU$, 
$Y_1=c_0X$ and $Y_{m+1}=C_0(C_1(Y_m\sim X\cap X)$. 
Then it is clear that $Y_m={s: \lambda\leq s_1}$. $|RgY|=|U|+1.$ The $\A$ is finite and 
is simple and generated by a single element. From the above we get that the free algebras have infinitely many atoms.
\end{example}

\begin{theorem}  For every finite $n>2$, and $\beta>0$, there is an irredundant $\beta$ 
element generator set of $\Fr_{\beta}\RCA_n$ which does not generate it freely. The same holds for Pinters substitution algebras.
\end{theorem}
\begin{proof} For the first part, we take $n=3$, which is the most difficult case, because the corresponding logic has the least number of 
variables. This part is due to N\;emeti, though to the best of our knowledge it was not published in this form, which is also due to Nemeti in a preprint of 
his. Let $\L$ be a language with $3$ variables, and one tenary relation. The formulas that we will construct will be 
restricted meanning that 
variables occur only in their natural order. 
We shall construct three restricted formulas 
$\phi$, $\psi$ and $\eta$ such that
$\models R(x,y,z)\longleftrightarrow \psi(R/\phi),$
$\models \eta(R/\phi)\text { but not }\models \eta.$
A restricted formula is one such that variables in its atomic subformulas occur only in their natural order.
Here $\psi(R/\phi)$ is the formula obtained from $\psi$ by replacing all occurances of $R$ with $\phi$
and $x,y,z$ are the variables $v_0, v_1, v_2$ of $\Lambda$ respectively. 
In the following we write $R$ instead of $R(x,y,z).$
We may write $Rxy$ for $R(x,y)$.
Let $$suc(x,y)=\forall z([Rzy\land z\neq y]\longleftrightarrow [Rxz\lor z=x])$$
If we look at $R$ as a binary relation symbol interpreted as an order, then $suc(x,y)$ says that $y$ is the element after $x$. 
$A$ is the following set of formulas 
$$[R\longleftrightarrow \exists zR, Rxy\land Ryx\longrightarrow x=y, x\neq y\longrightarrow (R(xy)\lor R(yz)),$$ 
$$(\forall x)(\exists y)suc(x,y), \exists y(Ryy\land \forall x[Rxx\longrightarrow Rxy])].$$
$$Ax=(\forall x y z)\bigwedge A$$
Now $Ax$ says that $R$ is binary, and is a discrete ordering without endpoints and has a greatest fixed point.  
Call such a relation good.
Let 
$$\phi=R\lor(Ax\land x=y\land \exists z[suc(x,z)\land Rzz\land (\forall x)(Rxx\longrightarrow Rxz)].$$
Now $\phi$ says that if $R$ is good then $\phi$ represents $\bar{R}$ where $\bar{R}$ is
$$R\cup \{ \text{ the successor of the greatest fixed point of $R$ as a new fixed point }\},$$ 
otherwise $\phi$ is $R$.
$$\psi=(\neg Ax\land R)\lor (Ax\land R\land [x=y\longrightarrow \exists y(x\neq y\land Rxy\land Ryy)]),$$
$\psi$ is $R$ without the greatest fixed point if $R$ is good, otherwise it is $R$.
$$\eta=Ax\longrightarrow \exists xy(x\neq y\land Rxx\land Ryy).$$
$\eta$ says that if $R$ is good then it has at least two fixed points.
Then  $\psi(R/\phi)$ is equivalent to $R$ since $R$ can be recovered from $\bar{R}$ by omitting its greatest fixed point.
$\eta(R/\phi)$ is true since if $R$ is good then $\bar{R}$ has two fixed points. Clearly for every infinite set $M$ there is a model
$\M$ with universe $M$ such that not $\M\models \eta.$
Then $\phi$, $\psi$ and $\eta$ are as required. 
By using the correspondance between terms and restricted formulas we obtain three terms $\tau(x)$, $\sigma(x)$
and $\delta (x)$ such that
$\RCA_3\models \sigma(\tau(x))=x$ and $\RCA_3\models \delta(\tau(x))=1$ but not $\Cs_3\models \delta(x)=1$.
Then for every $n\geq 3$ we have 
(a) $\RCA_n\models \sigma(\tau(x))=x$ 
and 
(b) $\RCA_n\models \delta(\tau(x))=1$ but not $\Cs_n\models \delta(x)=1.$
Let $0<\beta$, and $n\geq 3$ and let $\{g_i:i<\beta\}$ be an arbitray generator set of $\Fr_{\beta}\RCA_n.$
Then $\{\tau(g_0)\}\cup \{g_i: 0<i<\beta\}$ generates $\Fr_{\beta}\RCA_n$ by (a) but not freely by (b).

For the second part, we add a binary relation to our language and we pretend that it the membership relation in set theory; 
in fact will be the real membership relation
when semantically  interpreted, which will be the case. 

The idea is to translate any formula with equality to one having an extra binary relation, that acts as equality
such that the two are equivalent modulo a certain strong 
congruence, and the second is equality free using the existentional axiom of set theory.

The proof is purely semantical, which makes life easier. However, there is a syntactical proof too, 
using the pairing technique of Tarski substantially modified by N\'emeti, giving the
same result for $SC_3$, but we omit this much more involved proof. This pairing technique, implemented via a recursive translation
function for $L_{\omega,\omega}$ to $L_3$ preserves meaning, hence providing a completeness theorem for $\CA_3$. (Larger $n$ is much easier,
below we will deal with paring function in dimension $4$ a technique invented by Tarski.) 
Such a procedure enables one to transfer results proved for the representable algebras to the abstract ones, 
though the distance between them is infinite, in some precise sense (a result of Monk).

Let $\Ax_{eq}$ and $\Ax_{cong}$ be as in \cite{a} and $\tr$ be the function that takes every formula to an equality free formula.
The latter is an adjoint function, and it clearly preserves meaning

For $\{x,y,z\}=\{v_0, v_1, v_2\}$, these are defined as follows:
$$\Ax_{eq}=\{\forall x\forall y(x=y\leftrightarrow (\forall z(z\in x\leftrightarrow z\in y))\}$$ 
$$\Ax_{cong}= \{\forall xy(\forall z(z\in x\leftrightarrow z\in y))\to (\forall z(x\in z\leftrightarrow y\in z))\}$$
For a formula $\phi$ with equality, $\tr(\phi)$ is obtained from $\phi$ by replacing all of the occurrences of $x=y$ by 
$\forall z(z\in x\leftrightarrow z\in y).$
Notice that such formulas can be defined by algebraic terms, the former in cylindric algebras and the second in Pinters algebras.
Then, we have  
$\Ax_{eq}\vdash \phi\leftrightarrow \tr(\phi).$

For $\M$ a model  for a language without equality, define the Leibniz congruence $\sim $ by 
$$a\sim b\longleftrightarrow \forall z(z\in a\leftrightarrow z\in b).$$
It is not hard to check that $\sim$ is a {\it strong} congruence; 
it preserves $\in$ in both directions.
 
Then for any formula with equality using 3 variables, and $\M$ a model without equatity of $\Ax_{cong}$, we have
$\M\models \tr(\phi)$ iff $\M/\sim \models \phi$. Notice that $\phi$ has equality, we have $a/\sim=b/\sim$ iff $a\sim b,$
so that is is meaningful to talk about equality here. 

Let $\psi$ be any formula with equality and let $\tr(\phi)$ be the equality free corresponding formula, using the membership binary relation.
A piece of notation: If $\M$ is a model for the language with equality, let $\A_M\in \Cs_n$ be the correponding set algebra, 
and same for models without equality; in this last case, we denote the corresponding Pinter's set algebra corresponding to 
$\N$ by $\B_N$. 

Let $\M$ be a model for the language with equality such that  $\M\models \Ax_{eq}$, then there exists 
$\N$ a model for the language without equality
such that $\M\cong \N/\sim$. Then there is an  an  induced base isomorphism betwen $\B_N\to \Rd_{sc}\A_M.$

Now we use the correspondence betwen formulas and terms, we lift the translation function to the level of terms.
If $\tau$  corresponds to $\psi$ then let
$\tr(\tau)$ be that corresponding to $\tr(\phi)$.
Frm the above we have 
and $\RCA_{\alpha}\models \tau =1$ iff $\RSC_{\alpha}\models tr(\tau)=1$.
and we done, from the first part of the  proof.
\end{proof}

\begin{theorem}
There is a formula $\psi\in L_4$ such that no consistent 
recursive extension $T$ of $\psi$ is complete, and moreover, $\psi$ is hereditory inseparable meaning that  
no recursive extension of $\psi$ separates the $\vdash$ consequences
of $\psi$ from the $\psi$ refutable sentences.
\end{theorem} 
\begin{proof} We assume that we have only one binary relation and we denote our language by $L_4(E,2)$. 
This is implicit in the Tarski Givant approach, 
when they interpreted $ZF$ in $\RA$. $L_4$ is very close to $\RA$ but not quite
$\RA$, it s a little bit stronger 9for eaxmple there are four variable terms that cnnot be expressed in $\RA$ terms). 
The technique is called the {\it pairing} technique, which uses quasi projections to code extra variable, 
establishing the completeness theorem
above for $\vdash_4$. 

We have one binary relation $E$ in our language; for convenience, 
we write $x\in y$ instead of $E(x,y)$, to remind ourselves that we are actually working in the language
of set theory. 
We define certain formulas culminating in formulating the axioms of a finite undecidable theory, better known as Robinson's arithmetic 
in our language. These formulas are taken from N\'emeti.
We  formulate the desired hereditory inseparable $\psi$ in $L_4(E,2)$.

For $4$ variables, we need the following 'translation' result of Tarski which states a basic 
property of Tarski's pairing functions, namely we can code up, 
or represent, any sequence of variables in terms of a single variable, thus effectively 
reducing the number of variables to one. In more detail, we have:

\begin{athm}{Fact}  Let $p_0(x,y)$ and $p_1(x,y)$ be in $L_3(E,2)$ and let
$$\pi=(\forall x)(\forall y)(\forall z)[(p_0(x,y)\land p_0(x,z)\implies y=z)\land$$ 
$$p_1(x,y)\land p_1(x,z)\implies y=z)\land$$
$$\exists z(p_0(z,x)\land p_1(z,y)].$$ 

be the formula stipulating that they are quasiprojections. 
Then there is a recursive function $\tr:L_{\omega}(E,2)\to L_3$
such that $(i)-(iii)$ below hold for every $\phi\in L_3(Em2)$
\begin{enumroman}
\item $\pi\models \phi\longleftrightarrow \tr\phi$
\item $\tr(\neg \phi)=\neg  \tr(\phi),$  
\end{enumroman}
\end{athm} 

First  we interpret usual Robinson arithmetic in the usual language with $\omega$ many variables, 
using the standard interpretation of Peano arithmetic into set theory relativized to finite hereditory 
sets (that is Peano arithmetic with axiom of infinity). 
This part is semantical, in nature, so it is not too difficult to implement:
$$x=\{y\}=:y\in x\land (\forall z)(z\in x\implies z=y)$$
$$\{x\}\in y=:\exists z(z=\{x\}\land z\in y)$$
$$x=\{\{y\}\}=:\exists z(z=\{y\}\land x=\{z\})$$
$$x\in \cup y:=\exists z(x\in z\land z\in y)$$
$$pair(x)=:\exists y[\{y\}\in x\land (\forall z)(\{z\}\in x\to z=y)]\land \forall zy[(z
\in \cup x\land \{z\}\notin x\land$$
$$y\in \cup x\land \{y\}\notin x\to z=y]\land \forall z\in x\exists y
(y\in z).$$
Now we define the pairing functions:
$$p_0(x,y)=:pair(x)\land \{y\}\in x$$
$$p_1(x,y)=:pair(x)\land [x=\{\{y\}\}\lor (\{y\}\notin x\land y\in \cup x)].$$
$p_0(x,y)$ and $p_1(x,y)$ are defined.

$$x\in Ord= : \text {`` $x$ is an ordinal, i.e. $x$ is transitive and $\in$ is a total ordering on $x$},$$
$$x\in Ford=:x\in Ord\land \text { ``every element of $x$ is a successor ordinal "}$$
$$\text { i.e. $x$ is a finite ordinal }.$$
$$x=0=: ``x\text { has no element }"$$
$$sx=z=:z=x\cup \{x\},$$
$$x\leq y=:x\subseteq y,$$
$$x<y=: x\leq y\land x\neq y,$$
$$x+y=z=:\exists v(z=x\cup v\land x\cap v=0 \land$$
$$\text {``there exists a bijection between $v$ and $y$"})$$
$$x\cdot y=z=:\text { ``there is a bijection between $z$ and $x\times y$}"$$
$$x\underline{exp} y=z: \text { there is a bijection between 
$z$ and the set of all functions from $y$ to $x$}"$$
Now $\lambda$' is the formula saying that:
$0, s, +, \cdot, \underline{exp}$ are functions of arities $0,1,2,2,2$ on $Ford$
and 
$$(\forall xy\in Ford)[sx\neq 0\land sx=sy\to x=y)\land (x<sy\longleftrightarrow x\leq y)
\land$$
$$\neg(x<0)\land (x<y\lor x=y\lor y<x)\land (x+0=x)\land (x+sy=s(x+y))\land (x.0=0)$$
$$\land (x\cdot sy=x\cdot y+x)\land (x\underline{exp} 0=s0)\land (x\underline{exp} sy=x\underline{exp} y\cdot x)].$$
Now the existence of the desired incompletable $\lambda$ readily follows:
$\lambda\in Fm_{\omega}^0$.
Let ${\sf RT}$ be the absolutely free relation algebra on one genetrator. Let $p=r(p_0(x,y))$ and $q=r(p_1(x,y))$, 
where $r$ is the recursive function mapping $L_3(E,2)$ into ${\sf RT}$ that also preserves meaning.
Here ${\sf RT}$ is the set terms in the language of relation algebras with only one generator.
$$\pi_{\RA}=(\breve{p};p\to Id)\cdot (\breve{q};q\to Id)\cdot (\breve{p};q).$$
This is just stipluation that $p$ and $r$ are quasi projections, in the language of relation algebras. 
Then, we have  $\pi_{\RA}\in {\sf RT}$ since $p_i(x,y)\in L_3(E,2)$

Let $\lambda\in L_{\omega}$ be the inseparable sentence, that is the conjuction of finite axioms of Robinson arithmetic)
constructed above and let
$\psi=(r(\tr(\lambda))\cdot \pi_{\RA}$. 
From the definition of $r$ and $\tr$ we have 
$\eta\in {\sf RT}$. Let $\Fm_4$ be the formula algebra built on $L_4(E,2)$.
Let $\cal G$ be the absolutely free $\RA$ algebra on one generator $g$.
Let $h: {\cal G} \to \Ra\Fm_4$ be the 
homomorphism that takes the free generator of $\cal G$ to $x\in y.$
Let $\psi=h(\eta)$. Then $\psi\in \Fm_4$ and furthermore, it can be checked that $\psi$ is the desired formula. 
(Here we use that the $\Ra$ reduct of a $\CA_4$ is a relation algebra.)
\end{proof}

The same idea can be implemented by avoiding the path from $\CA$ to $\RA$ and then back to $\CA$, using Simon's result, 
and a very deep result of N\'emeti's.

N\'emeti defines a set of axioms ${\sf Ax}$ that is semantically equivalent to $\pi$ but stronger (proof theoretically).
The idea to translate all $3$ variable usual first order formula into the $\QRA$ fragment of $L_3(E,2)$. 
We have the quasiprojections $p_0$, 
$p_1$ and the set of axioms ${\sf Ax}$; which say that $p_0$ and $p_1$ are quasiprojections, and it can 
prove a strong form of associativity of relations.

We also know that in very $\QRA$, for each $n\in \omega$, there sits in a $\CA_n$, 
and there are cylindric algebras of various increasing finite dimensions
synchronized by the neat reduct functor, so that
the $\CA_3$ sitting there, has the cylindric neat embedding property 
(it neatly embeds, and indeed faithfully so in cylindric algebras of arbitary larger finite 
dimensions theorem). Yet again, by Henkin's neat embedding thoerem,  this algebra call it $\C$ is representable.

Define $f: \Fm_3\to \C$ be the homomorphism defined the usual way. 
Then define the translation map as follows:
$\tr(\phi)={\sf Ax}\to f(\phi)$. This functions covers the infinite gap between $\vdash_3$ and $\models.$
The above proof for $3$ dimensions can be done by $\pi$ instead of ${\sf Ax}$.

For cylindric algebras, diagonal free cylindric algebras Pinter's algebras and quasipolyadic equality, though free algebras of $>2$ dimensions
contain infinitely many atoms, they are not atomic. 
(The diagonal free case of cylindric algebras is a very recent result, due to Andr\'eka and N\'emeti, that has profound
repercussions on the foundation of mathematics.)
We, next, state two theorems that hold for such algebras, in the general context of $BAO$'s. But first a definition.

\begin{definition} Let $K$ be a class of $BAO$ with operators $(f_i: i\in I)$
Let $\A\in K$. An element $b\in A$ is called {\it hereditary closed} if for all $x\leq b$, $f_i(x)=x$.
\end{definition}
In the presence of diagonal elements $d_{ij}$ and cylindrifications $c_i$ for indices $<2$, $-c_0-d_{01},$ 
is hereditory closed. 
\begin{theorem}
\begin{enumarab}
\item Let $\A=\Sg X$ and $|X|<\omega$. Let $b\in \A$ be hereditary closed. Then $\At\A\cap \Rl_{b}\A\leq 2^{n}$.
If $\A$ is freely generated by $X$, then $\At\A\cap \Rl_{b}\A= 2^{n}.$
\item If every atom of $\A$ is below $b,$ then $\A\cong \Rl_{b}\A\times \Rl_{-b}\A$, and $|\Rl_{b}\A|=2^{2^n}$.
If in addition $\A$ is infinite, then $\Rl_{-b}\A$ is atomless.
\end{enumarab}
\end{theorem}
\begin{proof} Assume that $|X|=m$. We have $|\At\A\cap \Rl_b\A|=|\{\prod Y \sim \sum(X\sim Y).b\}\sim \{0\}|\leq {}^{m}2.$
Let $\B=\Rl_b\A$. Then $\B=\Sg^{\B}\{x_i.b: i<m\}=\Sg^{Bl\B}\{x_i.b:i<\beta\}$ since $b$ is hereditary fixed.
For $\Gamma\subseteq m$, let
$$x_{\Gamma}=\prod_{i\in \Gamma}(x_i.b).\prod_{i\in m\sim \Gamma}(x_i.-b).$$
Let $\C$ be the two element algebra. Then for each $\Gamma\subseteq m$, there is a homomorphism $f:\A\to \C$ such that
$fx_i=1$ iff $i\in \Gamma$.This shows that $x_{\Gamma}\neq 0$ for every $\Gamma\subseteq m$,
while it is easily seen that $x_{\Gamma}$ and $x_{\Delta}$
are distinct for distinct $\Gamma, \Delta\subseteq m$. We show that $\A\cong \Rl_{b}\A\times \Rl_{-b}\A$.

Let $\B_0=\Rl_{b}\A$ and $\B_1=\Rl_{-b}\A$.
Let $t_0=b$ and $t_1=-b$. Let  $J_i=\Ig\{t_i\}$
Assume that $z\in J_0\cap J_1$. Then $z\leq t_i$,
for $i=0, 1$, and so  $z=0$. Thus $J_0\cap J_1=\{0\}$. Let $y\in A\times A$,
and let $z=(y_0.t_0+y_1.t_1)$, then $y_i.x_i=z.x_i$ for each $i=\{0,1\}$
and so $z\in \bigcap y_0/J_0\cap y_1/J_1$. Thus $\A/J_i\cong \B_i$, and so
$\A\cong \B_0\times \B_1$.

\end{proof}
The above theorem holds for free cylindric and quasi-polyadic equality algebras. The second part  (all atoms are zero-dimensional) 
is proved by Mad\'arasz and N\'emeti.

The following theorem holds for any class of $BAO$'s.
\begin{theorem}\label{atomless}
The free algebra on an infinite generating set is atomless.
\end{theorem}
\begin{proof} Let $X$ be the infinite freely generating set. Let $a\in A$ be non-zero. 
Then there is a finite set $Y\subseteq X$ such that $a\in \Sg^{\A} Y$. Let $y\in X\sim Y$.
Then by freeness, there exist homomorphisms $f:\A\to \B$ and $h:\A\to \B$ such that $f(\mu)=h(\mu) $ for all $\mu\in Y$ while
$f(y)=1$ and $h(y)=0$. Then $f(a)=h(a)=a$. Hence $f(a.y)=h(a.-y)=a\neq 0$ and so $a.y\neq 0$ and
$a.-y\neq 0$.
Thus $a$ cannot be an atom.
\end{proof}

\end{document}